\documentclass{endm}
\usepackage{endmmacro}
\usepackage{graphicx}
\usepackage[paper=a4paper, top=4.5cm]{geometry}

\begin{document}

\begin{verbatim}\end{verbatim}\vspace{2.5cm}

\begin{frontmatter}

\title{Maximizing the number of vertices of an $N$-clique cover of the edges of a graph on $N$ vertices}

\author{Leopoldo Taravilse}
\address{Departamento de Computaci\'on\\ Universidad de Buenos Aires\\ Buenos Aires, Argentina}

\begin{abstract}
All the work made so far on edge-covering a graph by cliques focus on finding the minimum number of cliques that cover the graph. On this paper, we fix the number of cliques that cover a graph by the same number of vertices that the graph has, and give an upper bound for the sum of the number of vertices of these cliques in the cases where this covering is possible.
\end{abstract}

\begin{keyword}
clique, cover, vertices, graph
\end{keyword}

\end{frontmatter}

\section{Introduction}\label{intro}

Several works have been made on the minimum edge-covering of a graph by a certain family of graphs. On \cite{cohen}, Cohen and Tarsi prove that deciding if a graph $G$ can be edge-covered with graphs isomorphic to a graph $H$ and finding the decomposition is NP-complete. On \cite{orlin}, Orlin presents the problem of edge-covering a graph $G$ by cliques, but he is interested on the minimum number of cliques that can cover the graph. On \cite{pullman} Pullman presentes a survey of all the work made on minimum edge-clique covering so far, and on \cite{roberts} Roberts gives applications to this problem.

\section{Proving the upper bound}

If we want to cover the edges of a graph $G$ of $N$ vertices, with $N$ edge-disjoint cliques, then we can say that this cliques have $V_1, V_2, ..., V_N$ vertices. Note that it is not always possible to cover the edges of a graph of $N$ vertices with $N$ edge-disjoint cliques, but we will consider only the cases where this covering is possible. The number of edges that this cliques have is in total

$$\displaystyle \sum_{i=0}^{N} \frac{V_i(V_i-1)}{2}$$

and we know that number cannot exceed $\frac{N(N-1)}{2}$, that is an upper bound for the number of edges of $G$. What we want to maximize is 

$$\displaystyle \sum_{i=0}^{N} V_i$$

We know that $V_i$ must be integers, but we can give an upper bound for the case $V_i \in \mathbb{R}_{>0}$ and then the bound holds if $V_i$ are integers.

\begin{theorem}
An upper bound $B(N)$ to the sum of the number of vertices of the $N$ cliques covering a graph on $N$ vertices is $Nf^{-1}(N-1)$ where $f(M) = M(M-1)$.
\end{theorem}

\begin{proof}

As we know that 

\begin{eqnarray*}
\displaystyle \sum_{i=0}^{N} \frac{V_i(V_i-1)}{2} &\leq& \frac{N(N-1)}{2}\\
\displaystyle \sum_{i=0}^{N} V_i(V_i-1) &\leq& N(N-1)
\end{eqnarray*}

must hold, let us replace each $V_i$ with the average of all of them. We can see that 

$$N\frac{\displaystyle \sum_{i=0}^{N} V_i}{N} \left(\frac{\displaystyle \sum_{i=0}^{N} V_i}{N}-1\right) \leq \displaystyle \sum_{i=0}^{N} V_i(V_i-1)$$

holds too. Let us prove this inequality:

$$N\frac{\displaystyle \sum_{i=0}^{N} V_i}{N} \left(\frac{\displaystyle \sum_{i=0}^{N} V_i}{N}-1\right) \leq \displaystyle \sum_{i=0}^{N} V_i(V_i-1)$$
$$\displaystyle \sum_{i=0}^{N} V_i \left(\frac{\displaystyle \sum_{i=0}^{N} V_i}{N}-1\right) \leq \displaystyle \sum_{i=0}^{N} V_i(V_i-1)$$
if and only if
$$\displaystyle \sum_{i=0}^{N} V_i \left(\frac{\displaystyle \sum_{i=0}^{N} V_i}{N}\right) \leq \displaystyle \sum_{i=0}^{N} V_i^2$$
if and only if
$$\left(\displaystyle \sum_{i=0}^{N} V_i\right)^2 \leq N\displaystyle \sum_{i=0}^{N} V_i^2$$

And this holds because of the Cauchy-Schwarz inequality. So it is enough to prove the upper bound for the mean value. Let us call $M$ to the mean value of the $V_i$s in this case, now we know that 

\begin{eqnarray*}
N \frac{M(M-1)}{2} &\leq& \frac{N(N-1)}{2}\\
M(M-1) &\leq& N-1\\
f(M) &\leq& N-1
\end{eqnarray*}

where $f(M) = M(M-1)$ as we defined before. Now we can conclude that if the maximum sum of the number of vertices of the $N$ cliques that edge-cover a graph $G$ on $N$ vertices is $T(N)$ then $T(N) \leq Nf^{-1}(N-1)$, so $B(N) = Nf^{-1}(N-1)$ is the upper bound that we wanted to find.

\end{proof}

\section{Values of $N$ where the bound holds}

Now we have an upperbound $B(N)$ for the sum of the number of vertices of the $N$ cliques, we will give examples where the bound is reached.

\begin{theorem}
If $P$ is a prime number, then for $N = P^2 + P + 1$ the bound can be reached taking $K_N$ as $G$.
\end{theorem}

\begin{proof}

Let us take a prime $P$, and let $N = P^2 + P + 1$, and let us organize the vertices of $K_N$ in $P+1$ groups of $P$ vertices indexed $v_{i,j}$ with $0 \leq i \leq P$ and $0 \leq j < P$, and a special vertex $w$. We will take $P^2 + P + 1$ edge disjoint $K_{P+1}$s.

\begin{itemize}
\item Type 1 cliques: For each $i_0$ we will have a $K_{P+1}$ composed by $w$ and every $v_{i_0,j}$ for a total of $P+1$ cliques.
\item Type 2 cliques: For each $0 \leq A, B < P$ we will have a clique that contains $v_{i,j}$ if and only if $0 \leq i < P$ and $Ai+j = B \textrm{ mod } P$ or $i = P$ and $j = A$. This is a total of $P^2$ cliques.
\end{itemize}

It is easy to see that the type 1 cliques are edge disjoint, and that no type 1 clique has an edge in common with a type 2 clique, now let us see that no type 2 cliques have two vertices in common.

Let us see that if $v_{i_1,j_1}$ and $v_{i_2,j_2}$ ($i_1 \neq i_2$ and $j_1 \neq j_2$) belong to $K_{A_1,B_1}$ and $K_{A_2,B_2}$ then $A_1 = A_2$ and $B_1 = B_2$. If $i_1 = P$ then $A_1 = A_2 = j_1$ and because $i_2 \neq P$ then $B_1 = B_2$. The same happens if $i_2 = P$.

In case $i_1, i_2 < P$ then we have 

\begin{eqnarray*}
A_1 i_1 + j_1 &=& B_1\\
A_2 i_1 + j_1 &=& B_2\\
A_1 i_2 + j_2 &=& B_1\\
A_2 i_2 + j_2 &=& B_2
\end{eqnarray*}

Hence

\begin{eqnarray*}
A_1 i_1 + j_1 &=& A_1 i_2 + j_2\\
A_2 i_1 + j_1 &=& A_2 i_2 + j_2\\
A_1 (i_1-i_2) &=& A_2 (i_1-i_2)
\end{eqnarray*}

And because $p$ is prime we can divide by the inverse of $(i_1-i_2)$ in both sides (as we know it is not 0) so $A_1 = A_2$ and then it must hold $B_1 = B_2$.

This way we found a complete graph on $P^2 + P + 1$ vertices with $P^2 + P + 1$ edge disjoint $K_{P+1}$ for a total of $(P^2 + P + 1) (P+1)$ vertices.

Now as $N = P^2 + P + 1$ and $M = (P+1)$ is the number of vertices per each clique we must prove that $M(M-1) = N-1$ but this trivially holds.

\end{proof}

\section{Approaching the bound for every $N$}

Now let us call $T(N)$ to the maximum sum of vertices of $N$ edge disjoint cliques edge covering a graph on $N$ vertices. Let us prove that $\displaystyle \lim_{N \to \infty} \frac{T(N)}{B(N)} = 1$. First we will use the following two lemmas:

\begin{lemma}
For every $\epsilon > 0$ there exists a positive integer $N_0$ such that if $N > N_0$ there is a prime $P$ such that $N(1-\epsilon) < P < N$.
\end{lemma}

\begin{proof}
The prime numbers theorem states that if $\pi(N)$ is the number of primes between 1 and $N$ then $\displaystyle \lim_{N \to \infty} \frac{\pi(N)}{\frac{N}{\ln N}} = 1$. Let us see that $\pi(N)-\pi(N(1-\epsilon))$ diverges so there is $N_0$ such that $\pi(N) - \pi(N(1-\epsilon)) > 1$ for every $N > N_0$.

\begin{eqnarray*}
\displaystyle \lim_{N\to\infty} \frac{\pi(N(1-\epsilon))}{\frac{N(1-\epsilon)}{\ln (N(1-\epsilon))}} &=& 1\\
\displaystyle \lim_{N\to\infty} \frac{\pi(N(1-\epsilon))}{\frac{N}{\ln (N(1-\epsilon))}} &=& 1-\epsilon\\
\displaystyle \lim_{N\to\infty} \frac{\pi(N(1-\epsilon))}{\frac{N}{\ln N} \frac{\ln N}{\ln N + \ln (1-\epsilon)}} &=& 1-\epsilon\\
\displaystyle \lim_{N\to\infty} \frac{\pi(N(1-\epsilon))}{\frac{N}{\ln N}} &=& 1-\epsilon\\
\displaystyle \lim_{N\to\infty} \frac{\pi(N) - \pi(N(1-\epsilon))}{\frac{N}{\ln N}} &=& \epsilon
\end{eqnarray*}

And as $\displaystyle \lim_{N\to\infty} \frac{N}{\ln N} = \infty$ then $\displaystyle \lim_{N\to\infty} \pi(N) - \pi(N(1-\epsilon)) = \infty$ too.

\end{proof}

\begin{lemma}
For every $\epsilon > 0$ there exists a positive integer $N_0$ such that if $N > N_0$ there is a prime $P$ such that $N(1-\epsilon) < P^2 + P + 1 < N$
\end{lemma}

\begin{proof}
Because of Lemma 1 we know that for every $\epsilon > 0$ there is a positive integer $i_0$ such that if $i > i_0$ and $P_i$ is the $i-th$ prime then $\frac{P_i}{P_{i+1}} > 1 - \frac{\epsilon}{2}$. Let us see that $\frac{P_i^2 + P_i + 1}{P_{i+1}^2 + P_{i+1} + 1} > 1 - \epsilon$.

We will use that $P_i > (1-\frac{\epsilon}{2}) P_{i+1}$ and that $(1-\frac{\epsilon}{2})^2 > 1 - \epsilon$

\begin{eqnarray*}
\frac{P_i^2 + P_i + 1}{P_{i+1}^2 + P_{i+1} + 1} &>& \frac{P_{i+1}^2 (1-\frac{\epsilon}{2})^2 + P_{i+1}(1-\frac{\epsilon}{2}) + 1}{P_{i+1}^2 + P_{i+1} + 1}\\
&>& \frac{P_{i+1}^2 (1-\epsilon) + P_{i+1}(1-\epsilon) + 1-\epsilon}{P_{i+1}^2 + P_{i+1} + 1}\\
&>& 1-\epsilon
\end{eqnarray*}

And this concludes our proof of Lemma 2.

\end{proof}

\begin{theorem} 
$\displaystyle \lim_{N \to \infty} \frac{T(N)}{B(N)} = 1$
\end{theorem}

\begin{proof}

Using Lemma 2 and given $\epsilon > 0$, we can take the value of $N_0$ given by this lemma for $\frac{\epsilon}{2}$, so we know that we can cover a graph on $N$ vertices with cliques whose vertices sum up to at least $B(N(1-\frac{\epsilon}{2}))$ as $B$ is an increasing function, and there is a prime $P$ such that $N(1-\frac{\epsilon}{2}) < P^2 + P + 1 < N$ so $B(N(1-\frac{\epsilon}{2})) < B(P^2+2P+1)$.

Now let us prove that 

$$B(N(1-\frac{\epsilon}{2})) > B(N) (1-\epsilon)$$

First of all we can see that $(M-1)^2 < f(M) < M^2$ so $\sqrt{N} < f^{-1}(N) < \sqrt{N}+1$. We can also see as $f'(M) = 2M-1$ that $(f^{-1})'(M(M-1)) = \frac{1}{2M-1}$ so as $\displaystyle \lim_{M \to \infty} (f^{-1})'(M(M-1)) = 0$ then $\displaystyle \lim_{N \to \infty} (f^{-1})'(N) = 0$.

Now we see that

\begin{eqnarray*}
B(N(1-\frac{\epsilon}{2})) &>& B(N) (1-\epsilon)\\
N (1-\frac{\epsilon}{2}) f^{-1}(N(1-\frac{\epsilon}{2})-1) &>& N f^{-1}(N-1) (1-\epsilon)\\
(1-\frac{\epsilon}{2}) f^{-1}(N(1-\frac{\epsilon}{2})-1) &>& f^{-1}(N-1) (1-\epsilon)
\end{eqnarray*}

As $\displaystyle \lim_{N \to \infty} (f^{-1})'(N) = 0$ then we can take $N$ large enough so that 

$$f^{-1}(N(1-\frac{\epsilon}{2})-1) > f^{-1}(N-1) (1-\frac{\epsilon}{2})$$

So we just need to prove that 

\begin{eqnarray*}
(1 - \frac{\epsilon}{2})^2 &>& 1 - \epsilon \\
1 - \epsilon + \epsilon^2 &>&= 1 - \epsilon\\
\epsilon^2 &>& 0
\end{eqnarray*}

And this is trivially true, so that concludes our proof.

\end{proof}

\end{document}